\documentclass{amsart}
\usepackage{amsmath}
\usepackage{amssymb}
\usepackage{enumerate}
\usepackage{amsthm}
\usepackage{wasysym}
\usepackage{stmaryrd}
\usepackage{mathrsfs}
\usepackage{pifont}
\usepackage{tikz}

\newcommand{\qee} {\hspace*{2mm}\hfill \ding{109}}

\renewcommand{\iff}{\leftrightarrow}
\renewcommand{\phi}{\varphi}

\newtheorem{theorem}{Theorem}[section]
\newtheorem{define}[theorem]{Definition}

\newtheorem{exa}[theorem]{Example}
\newenvironment{example}{\begin{exa} \rm}{\qee\end{exa}}
\newtheorem{exerc}[theorem]{Exercise}

\newtheorem{conj}[theorem]{Conjecture}

\newtheorem{ques}[theorem]{Open Question}
\newenvironment{question}{\begin{ques} \rm}{\qee\end{ques}}
\newtheorem{lem}[theorem]{Lemma}
\newenvironment{lemma}{\begin{lem} \it}{\end{lem}}
\newtheorem{cor}[theorem]{Corollary}

\newtheorem{rem}[theorem]{Remark}
\newenvironment{remark}{\begin{rem} \rm}{\qee\end{rem}}

 \newcommand{\tupel}[1]{{\langle #1 \rangle}}
 \newcommand{\verz}[1]{\{ #1 \}}

 \newcommand{\nrhd}{\mathrel{\not\! \rhd}}

\newcommand{\sline}{\raise-0.3ex\hbox{$\hbox{--}\kern-0.84ex\raise0.45ex\hbox{$\hbox{\scalebox{0.3}{\bf /}}
\kern-0.37ex\hbox{\scalebox{0.3}{\bf /}}$}$}}
\newcommand{\slinei}{\raise-0.3ex\hbox{$\hbox{--}\kern-0.84ex\raise0.45ex\hbox{$\hbox{\scalebox{0.3}{\bf \textbackslash}}
\kern-0.37ex\hbox{\scalebox{0.3}{\bf \textbackslash}}$}$}}
\newcommand{\lhdnneq}{\mathrel{\lhd_{\hspace*{-0.27cm}{}_{\kern0.2ex \slinei}}\hspace*{0.09cm}}}
\newcommand{\rhdnneq}{\mathrel{\rhd_{\hspace*{-0.27cm}{}_{\kern0.3ex \sline}}\hspace*{0.09cm}}}
\newcommand{\setvara}{{\mathcal A}}
\newcommand{\setvarb}{{\mathcal B}}
\newcommand{\setvarc}{{\mathcal C}}
\newcommand{\setvard}{{\mathcal D}}

\title[No minimal essentially undecidable Theories]{There are no\\ minimal essentially undecidable Theories}
\author{Fedor Pakhomov}\thanks{Research of Fedor Pakhomov was supported by FWO grant G0F8421N}
 \address{Vakgroep Wiskunde: Analysis, Logic and Discrete Mathematics, 
Ghent University,
Krijgslaan 281,
B9000~~Ghent,
Belgium\newline
and Steklov Mathematical Institute of Russian Academy of Sciences, 
Gubkina 8,
119991 Moscow, Russia}
\email{fedor.pakhomov@ugent.be}

\author{Juvenal Murwanashyaka}
 \address{Department of Mathematics, the Faculty of Mathematics and Natural Sciences,
                University of Oslo,
               Moltke Moes vei 35,
Niels Henrik Abels hus,
0851 Oslo, Norway}
\email{juvenalm@math.uio.no}
\date{\today}

\author{Albert Visser}
 \address{Philosophy, Faculty of Humanities,
                Utrecht University,
               Janskerkhof 13,
                3512BL~~Utrecht, The Netherlands}
\email{a.visser@uu.nl}
\date{\today}

\begin{document}



\thanks{We thank Yong Cheng and Tim Button for comments on and their corrections to the preprint version of this paper.}

\begin{abstract}
We show that there is no theory that is minimal with respect to interpretability among recursively enumerable  essentially undecidable theories.
\end{abstract}

\maketitle

\section{Introduction}
For any salient property of  recursively enumerable theories $\mathcal P$, one can ask the obvious question
\emph{is there a weakest recursively enumerable 
theory satisfying $\mathcal P$?} But what does
\emph{weakest} mean here? A traditional answer is to take a theory $T$ to be given by
a recursively enumerable set of axioms $\mathcal A$. \emph{Weakest} is then interpreted as: $T$ has $\mathcal P$ and
no theory axiomatised by a proper subset of $\mathcal A$ has $\mathcal P$.

A paradigmatic example of an answer to our question under this reading, for the case where we take
$\mathcal P$ to be \emph{essential undecidability},
 is the well-known result by Tarski, Mostowski and Robinson that the theory {\sf Q}
with its standard axiomatisation is minimally essentially undecidable in the sense that all
theories given by a proper subset of the axioms have a decidable extension. See \cite[Chapter 2, Theorem 11]{tars:unde53}.
Similar results for theories of concatenation were obtained by Juvenal Murwanashyaka.
See \cite{murw:weak22}.
For the theory {\sf R}, a result in the same spirit is due to Cobham. 
See \cite{jone:vari83}. Only here minimality is applied to natural \emph{groups} of axioms rather than to single axioms. 

The above results crucially depend on the chosen axiom set. After all, each non-trivial finitely 
axiomatisable theory is axiomatisable by a single
axiom. Suppose pure predicate logic does not satisfy $\mathcal P$. If a theory has property $\mathcal P$, 
then it is automatically minimal with respect to the single-axiom
axiomatisation.  Along a different line, for example, in the case of {\sf Q}, it is easy to produce finitely axiomatisable strict 
sub-theories that are still essentially
undecidable, if one allows tampering with the axioms. 
For example, we may relativise the quantifiers in the axioms for plus and times to the class of $x$ such that
${\sf S}x \neq x$. Jones and Shepherdson, in \cite{jone:vari83}, provide an example of an essentially undecidable sub-theory of
{\sf R} that is strictly below {\sf R} in the sense that it proves strictly less theorems. However, they do strengthen one
axiom group in order to be able to drop another. 

In this paper, we zoom in on the property of \emph{essential undecidability of recursively enumerable theories} 
and we consider an ordering of recursively enumerable theories that only depends on the 
theory-qua-set-of-theorems, to wit \emph{interpretability}.
So, we ask whether there is an interpretability-minimal recursively enumerable essentially undecidable theory.
In this paper we show that there is no such theory. Thus, our main result is:

\begin{theorem}\label{grotesmurf}
There is no interpretability-minimal recursively enumerable essentially undecidable theory.
\end{theorem}

It is easy to see that, if a recursively enumerable theory is \emph{minimally} essentially undecidable with 
respect to interpretability, then it is, \emph{ipso facto}, the \emph{minimum} recursively enumerable theory, 
modulo mutual interpretability.
This is because essentially undecidable recursively enumerable theories are closed
under finite infima. Thus, it is sufficient to prove the following lemma.
\begin{lemma}\label{hulpsmurf}
 There is no  recursively enumerable essentially undecidable
 theory that is the interpretability-minimum.
 \end{lemma}
 
 The paper is structured as follows.
 In Section~\ref{baco} we introduce the basic concepts needed for the paper as a whole. 
 This section also contains the reduction of Theorem~\ref{grotesmurf} to Lemma~\ref{hulpsmurf}.
 In Section~\ref{finax} we give a simple proof
 that there is no minimal theory with respect to interpretability among \emph{finitely axiomatised theories}. In Section~\ref{mare} we prove Theorem~\ref{grotesmurf}. 
 We provide two different proofs. The first proof employs a direct diagonalisation argument, the second reduces our question to
 a general recursion theoretic result.
 
\section{Some Basic Concepts}\label{baco}
Theories in the present paper are one-sorted recursively enumerable theories of predicate logic 
in finite signature. Our results are about theories-qua-sets-of-theorems. However, our methods 
sometimes demand the intensional perspective where the axiom set is given by a formula or by a recursive index.

There is a whole range of notions of interpretability.
We can or cannot have pieces, multidimensionality, parameters, relativisation,
non-preservation of identity. 
Our main result is not sensitive to the specific details of the notion
of interpretation as long as we consider mixtures of these specific features.

We want to prove that there is no interpretability-minimal recursively enumerable essentially 
undecidable theory. The proof can be divided in a preliminary step and a main step.
\begin{enumerate}[I.]
    \item 
   We show that, modulo mutual interpretability, 
   if there is an interpretability-minimal recursively enumerable essentially 
  undecidable theory, then there is an interpretability-minimum
 recursively enumerable essentially undecidable theory.
 To do this we show (i) that interpretability is a lower semi-lattice (modulo
 mutual interpretability), i.e., binary infima exist
 and (ii) that  infima preserve essential undecidability.
 This gives the reduction of Theorem~\ref{grotesmurf} to Lemma~\ref{hulpsmurf}.
 \item
 We prove Lemma~\ref{hulpsmurf}.
\end{enumerate}

Below we will define one uniform construction on theories defining the
binary minimum operation. This construction provides a binary minimum for all notions of
interpretability that are given by combinations of the above features.
So, (I) works for all notions of interpretability we under consideration.

It is immediate that, if we have (II) for the most inclusive notion,
where we have pieces, multidimensionality, parameters, relativisation, 
non-preservation of identity,
then we also have it for all weaker notions. So, we only need to prove
(II) for our strongest notion of interpretability. Alternatively,
one could look at what is used in the proof of (II) and see that
all notions satisfy these assumptions.

In this paper, we will use the most inclusive notion, to wit
piecewise, multidimensional (with dimensions varying over pieces),
relative, non-identity-preserving interpretability with parameters. 

We refer the reader for definitions to \cite{viss:onq17}.
Here we will just fix notations and give some basic facts.

We write $U \rhd V$ for $U$ interprets $V$ and $V \lhd U$ for
$V$ is interpretable in $U$.

Given two theories 
$U$ and $V$ we form $W := U \ovee V$ in the following way.
The signature of $W$ is the disjoint union of the signatures
of $U$ and $V$ with an additional fresh zero-ary predicate $P$.
The theory $W$ is axiomatised by the axioms $P \to \phi$ if $\phi$ is
a $U$-axiom and $\neg\, P \to \psi$ if $\psi$ is a $V$-axiom. One
can show that $U\ovee V$ is the infimum of $U$ and $V$ in the 
interpretability ordering $\lhd$. This result works for all
choices of our notion of interpretation.\footnote{It is a bit
strange that we use a disjunction-like notation for an
infimum. This strangeness is due to the fact that the
conventional choice for the interpretability-ordering
puts the weakest theory below where in boolean algebras
the strongest proposition is the lower one. In our notation,
we follow the boolean intuition and view our operation as
a kind of disjunction of theories.}
We note that $\ovee$ preserves finite axiomatisability.

A theory is \emph{essentially undecidable} iff all its consistent extensions in the
same language are undecidable. Salient examples of  essentially undecidable theories are {\sf R} and {\sf Q}. See \cite{tars:unde53}.
We have the following basic insights.

\begin{theorem}\label{ess_undec_inter_1}
$U$ is essentially undecidable iff every consistent $V$ such that $V \rhd U$,
 is undecidable.
\end{theorem}
\begin{proof} Since any $V$ extending $U$ also interprets $U$, the ``if'' part is trivial. For ``only if'' part we 
note that if $V\rhd U$ and $V$ is consistent and decidable, then any interpretation $\iota\colon V\rhd U$ gives 
a decidable consistent extension $T$ of $U$ that is axiomatized by all the sentences $\varphi$ of the language of $U$ 
such that $V$ proves the $\iota$-translation of $\varphi$.
\end{proof}

\begin{theorem}\label{ess_undec_inter_2}
We have:
\begin{enumerate}[i.]
    \item If $U$ is essentially undecidable and $V\rhd U$, then
    $V$ is essentially undecidable.
    \item If $U$ and $V$ are essentially undecidable, then 
    so is $U\ovee V$.
\end{enumerate}
\end{theorem}
\begin{proof}
The first claim follows immediately from Theorem \ref{ess_undec_inter_1} and the transitivity of interpretability.

To prove the second claim we assume for a contradiction that there is a consistent decidable extension 
$T$ of $U\ovee V$. Either $T$ is consistent with $P$ or with $\lnot P$. If $T$ is consistent with $P$, then $
T+P$ is a consistent decidable theory that interprets $U$, contradicting essential undecidability of $U$. 
Analogously, if $T$ is consistent with $\lnot P$, then $T+\lnot P$ is a consistent decidable theory that 
interprets $V$, contradicting essential undecidability of $V$.
\end{proof}

We note that the fact that essentially undecidable theories
are closed under binary infima implies that, if there is an
interpretability-minimal one, then there is a minimum 
with respect to interpretability. 
Thus, we have reduced Theorem~\ref{grotesmurf} to Lemma~\ref{hulpsmurf}.

This reduction still works when
we restrict ourselves to finitely axiomatised theories, since $\ovee$ also preserves 
finite axiomatisability..

\section{Finitely axiomatisable Theories}\label{finax}
In this section, we prove the non-existence of a minimal
essentially undecidable theory with respect to interpretability
for the case where we restrict ourselves to finitely axiomatised
theories.

Our result is really a triviality as soon as the required 
machinery is in place. We introduce this machinery in the
next subsection.

\subsection{Theories of a Number}
We need the theory {\sf TN} of a number. This theory is given as follows.
\begin{enumerate}[{\sf TN}1.]
\item
$\vdash x \not < 0$
\item
$\vdash (x< y \wedge y <z ) \to x < z$
%\item
%$\vdash x \not < x$ 
\item
$\vdash x< y \vee x= y \vee y < x$
%\item
%$\vdash {\sf S}x \neq 0$
\item
$\vdash x=0 \vee \exists y\, x={\sf S}y$
%\item
%$\vdash (x <z \wedge y <z) \to ({\sf S}x = {\sf S}y \to x=y)$ 
\item
$\vdash{\sf S}x \not < x$
\item
$\vdash x<y \to (x <{\sf S}x \wedge y \not < {\sf S}x)$
\item
$\vdash x+0=x$
\item
$\vdash x+{\sf S}y ={\sf S}(x+y)$
\item
$\vdash x \cdot 0 = 0$
\item
$\vdash x\cdot {\sf S}y = x\cdot y + x$
\end{enumerate}

\noindent
We note that {\sf TN} allows finite models which can be identified with the natural numbers viewed as the finite von Neumann ordinals with
the added structure of zero, addition, multiplication and $<$.
Moreover, if a model of {\sf TN} has a maximal element, the model can be viewed  as a (possibly) non-standard number.
In this case we take ${\sf S}a=a$ on the maximal element and adapt plus and times accordingly.
In \cite{viss:onq17}, the reader may find some further discussion of {\sf TN}.

A $\Delta_0$-formula is \emph{pure} iff (i) all bounding terms are variables and
(ii) all occurrences of terms are in subformulas of the form ${\sf S}x=y$, $x+y=z$ and $x \cdot y = z$.
A $\Sigma_1$-sentence \emph{pure} if it is of the form $\exists \vec x\; \sigma_0\vec x$, where
$\sigma_0$ is a pure $\Delta_0$-formula. 

We can transform an arbritrary $\Sigma_1$-sentence $\sigma$ into
a pure $\Sigma_1$-sentence. See \cite{viss:onq17}, for a
sketch of the argument. In Section~\ref{finax}, we will
assume that all $\Sigma_1$-sentences $\sigma$ are rewritten in pure form.

Let $\sigma := \exists \vec y\; \sigma_0\vec y$, where $\sigma_0$ is a pure $\Delta_0$-formula. We define:
\[ [\sigma] := {\sf TN} + \exists x \, \exists \vec y < x\, \sigma_0\vec y.\] 

We note that if $\sigma$ is false, then $[\sigma]$ extends {\sf R}. Thus, if $[\sigma]$ is, in addition,
consistent, we find that $[\sigma]$ is essentially undecidable.

\subsection{The main Result for finitely axiomatisable Theories}
We prove our main result for the finitely axiomatised case.

\begin{theorem}
There is no interpretability minimal essentially undecidable finitely axiomatised theory.
\end{theorem}

\begin{proof}
Since the finitely axiomatisable essentially undecidable theories are closed under interpretability-infima,
it is sufficient to show that there is no minimum theory $A^\star$
among finitely axiomatised essentially undecidable theories.

Suppose there was such an $A^\star$. Consider any $\Sigma^0_1$-sentence $\sigma$.
If $\sigma$ is true, then $[\sigma]$ has a finite model, so, clearly, $[\sigma] \nrhd A^\star$.
If $\sigma$ is false and $[\sigma]$ is consistent, we have $[\sigma]$ is essentially undecidable, so
$[\sigma] \rhd A^\star$. If $[\sigma]$ is inconsistent, then, trivially, $[\sigma] \rhd A^\star$.
\emph{Ergo}, $\sigma$ is false iff $[\sigma] \rhd A^\star$. However, this is impossible, since the set of
$\sigma$ such that
$[\sigma] \rhd A^\star$ is recursively enumerable.
\end{proof}

We note that our proof uses very little about interpretability.
So there is a good chance that it will work for even more general notions, like forcing-interpretability. However, we did not explore this.

\section{Recursively Enumerable Theories}\label{mare}
In this section we prove our main result. 

\subsection{A Result by Janiczak}
In this subsection, we present a basic result by Janiczak.
See \cite{jani:unde53}. This result will be the main tool for both proofs of Lemma~\ref{hulpsmurf}.

We formulate an immediate consequence of the results of Section 3 of \cite{jani:unde53}.
Let {\sf J} be the theory in the language with one binary relationsymbol {\sf E} with the following (sets of) axioms.\footnote{Our theory
differs slightly from the theory considered by Janiczak in that we added {\sf J}3. We did this to make the characterisation in
Theorem~\ref{janicz} as simple as possible.}
\begin{enumerate}[{\sf J}1.]
\item
{\sf E} is an equivalence relation.
\item
There is at most one equivalence class of size precisely $n$
\item
There are at least $n$ equivalence classes with at least $n$ elements.
\end{enumerate}

We define ${\sf A}_n$ to be the sentence: there exists an equivalence class of size
precisely $n+1$. It is immediate that the ${\sf A}_n$ are mutually independent over {\sf J}.

\begin{theorem}[Janiczak]\label{janicz}
Over {\sf J}, every sentence is equivalent with a boolean combination of the ${\sf A}_n$.
\end{theorem}
\begin{proof}
A base $n\ge 1$ configuration $C(\vec{x}\,)$ of variables $\vec{x}$ consists 
of two equivalence relations ${=_C}\subseteq \mathsf{E}_C$ on the set $\{\vec{x}\,\}$ of 
all variables from $\vec{x}$ together with a set ${\sf S}_C\subseteq \{1,\ldots,n-1\}$ and a 
function ${\sf s}_C\colon \{\vec{x}\}\to {\sf S}_C\cup \{n\}$, such that we have:
\begin{itemize}
\item
$x \mathrel{\sf E}_C x' \Rightarrow {\sf s}_C(x)={\sf s}_C(x')$;
\item
for any $x\in \{\vec{x}\,\}$, the
$ \mathrel{\sf E}_C$-equivalence class of $x$ is split into at most ${\sf s}_C(x)$-many 
$=_C$-equivalence classes;
\item
for any $i\in {\sf S}_C$, its ${\sf s}_C$-preimage is 
either empty or a single $\mathsf{\sf E}_C$-equivalence class. 
\end{itemize}
Working in 
$\mathsf{J}$, we say that $\vec{x}$ satisfies the configuration $C(\vec{x}\,)$ 
iff we have the following:
\begin{itemize}
    \item 
the restriction of $=$ to $\{\vec{x}\,\}$ is $=_C$  
and
the restriction of {\sf E} to $\{\vec{x}\,\}$ is ${\sf E}_C$
(this includes both positive and negative information);
\item
for each $i\in \{1,\ldots,n-1\}$, we have: $i$ is in ${\sf S}_C$ iff there is an $\mathsf{\sf E}$-equivalence class 
of the size $i$;
\item
for each $x\in \{\vec{x}\,\}$, if ${\sf s}_C(x)<n$, then $x$ is in the 
$\mathsf{E}$-equivalence class of the size ${\sf s}_C(x)$;
\item
for each $x\in \{\vec{x}\,\}$, if ${\sf s}_C(x)=n$, 
then $x$ is in an $\mathsf{E}$-equivalence class of the size $\ge n$.
\end{itemize}
We denote the set of all base $n$ configurations depending on the variables 
$\vec{x}$ as $\mathfrak{C}_n(\vec{x})$. Clearly, we can express the 
fact that $\vec{x}$ satisfies configuration $C(\vec{x})$ by a 
$\mathsf{J}$-formula which we will confuse with $C(\vec{x})$. 

Now observe the following facts:
\begin{enumerate}[1.]
 \item \label{disj_config} $\mathsf{J}\vdash \bigvee\limits_{C(\vec{x})\in \mathfrak{C}_n(\vec{x}\,)}C(\vec{x}\,)$.
 \item \label{dist_config} For distinct $C_1(\vec{x}\,),C_2(\vec{x}\,)\in \mathfrak{C}_n(\vec{x}\,)$ we have 
 $\mathsf{J}\vdash \lnot\, (C_1(\vec{x}\,)\land C_2(\vec{x}\,))$.
 \item \label{neg-config} For each $C(\vec{x}\,)\in \mathfrak{C}_n(\vec{x}\,)$,
 we have $\mathsf{J}\vdash \lnot\, C(\vec{x}\,) \leftrightarrow \bigvee\limits_{C'(\vec{x}\,)\in  \mathfrak{C}_n(\vec{x}\,)\setminus\{C(\vec{x}\,)\}}C'(\vec{x}\,)$.\\
 This follows from (\ref{disj_config}) and (\ref{dist_config}).
 \item \label{config-imp-atom} Given an
 atomic $\varphi(\vec{x}\,)$ and $C(\vec{x}\,)\in \mathfrak{C}_n(\vec{x})$, we have
 either $\mathsf{J}\vdash C(\vec{x}\,)\to \varphi(\vec{x}\,)$ or $\mathsf{J}\vdash C(\vec{x}\,)\to \lnot\, \varphi(\vec{x}\,)$.
 \item \label{config-imp-q-f}
 Given quantifier-free $\varphi(\vec{x}\,)$ and $C(\vec{x})\in \mathfrak{C}_n(\vec{x}\,)$, 
 either $\mathsf{J}\vdash C(\vec{x}\,)\to \varphi(\vec{x}\,)$ or $\mathsf{J}\vdash C(\vec{x}\,)\to \lnot\, \varphi(\vec{x}\,)$.
 This follows from (\ref{config-imp-atom}).
 \item \label{q-f-eq-config} For any quantifier-free $\varphi(\vec{x}\,)$ and $n\ge 1$,
 there is $\mathcal C\subseteq \mathfrak{C}_n(\vec{x})$, such that 
 $\mathsf{J}\vdash \varphi(\vec{x}\,)\leftrightarrow \bigvee\limits_{C(\vec{x}\,)\in \mathcal C}C(\vec{x}\,)$.
 This follows from (\ref{disj_config}) and (\ref{config-imp-q-f}).
 \item \label{q-elim-J} For  $C(\vec{x},y)\in \mathfrak{C}_n(\vec{x},y)$,
 let $\exists_yC(\vec{x},y) \in \mathfrak{C}_n(\vec{x})$ be the configuration $C'$ where 
 $S_{C'}=S_C$, and $=_{C'},\mathsf{E}_{C'},s_{C'}$ are the restrictions of 
 $=_C,\mathsf{E}_C,s_C$, respectively, to the new set of variables. If the size of $\vec{x}$ is at most $n-1$, then we have
  $\mathsf{J}\vdash \exists_yC(\vec{x},y)\leftrightarrow \exists y\, C(\vec{x},y)$.
\end{enumerate}

We claim that for any formula $\varphi(\vec{x}\,)$ of the language of $\mathsf{J}$ 
there is an $n$ and a $\mathcal C\subseteq \mathfrak{C}_n(\vec{x}\,)$,
such that \[\mathsf{J}\vdash \varphi(\vec{x}\,)\leftrightarrow \bigvee\limits_{C(\vec{x}\,)\in \mathcal C}C(\vec{x}\,).\]
For this, we put $\varphi(\vec{x}\,)$ in prenex normal form 
$Q_1 y_1,\ldots, Q_m y_m\, \psi(\vec{x},y_1,\ldots,y_m)$. 
We fix $n$ to be the sum of $m$ and the size of $\vec{x}$. We first use the fact (\ref{q-f-eq-config}) to transform the matrix $\psi(\vec{x},\vec{y}\,)$ into the disjunction of configurations from $\mathfrak{C}_n(\vec{x},y_1,\ldots,y_m)$. 
And, secondly, using facts (\ref{neg-config}) and (\ref{q-elim-J}),
we prove, by induction on $0\le i<m$, that $Q_{m-i} y_{m-i}\ldots Q_m y_m\, \psi(\vec{x},y_1,\ldots,y_m)$ is equivalent to a disjunction of configurations from $\mathfrak{C}_n(\vec{x},y_1,\ldots,y_{m-i})$. In the case of $i=m-1$ we obtain the desired representation for $\varphi(\vec{x}\,)$.

Combining the claim above with the observation that any $C()\in \mathfrak{C}_n()$ is $\mathsf{J}$-provably equivalent to a formula of the
form
$\bigwedge\limits_{i\in {\sf S}_C}\mathsf{A}_{i-1}\wedge \bigwedge\limits_{i\in \verz{1,\dots n-1}\setminus {\sf S}_C}\neg\, \mathsf{A}_{i-1}$, 
we see that, over $\mathsf{J}$, any sentence is equivalent to a boolean combination of the $\mathsf{A}_n$. 
\end{proof}

\begin{remark}
Let $\mathsf{B}_n(x)$ be $\mathsf{J}$-formula expressing that $x$ is in an equivalence class of size $>n$. The proof above, in fact shows that $\mathsf{J}$ enjoys quantifier elimination in the signature extended by all $\mathsf{A}_n$ and $\mathsf{B}_n(x)$. Indeed, it is easy to see that over $\mathsf{J}$ any configuration $C(\vec{x}\,)$ (as defined in the proof of Theorem \ref{janicz}) is equivalent to a quantifier-free formula in the extended signature and, in this proof,
we established that over $\mathsf{J}$ any formula is equivalent to a disjunction of configurations.
\end{remark}

\begin{remark}
Theorem~\ref{janicz} tells us that the Lindenbaum algebra of {\sf J} is isomorphic with respect to recursive boolean isomorphisms to
the free boole algebra on countably many generators. We note that, by the results of Pour-El \&\ Kripke~\cite[Theorem~2]{pour:dedu67},
the Lindenbaum algebra of any consistent RE theory $U$ that interprets the Tarski-Mostowski-Robinson theory {\sf R} is
recursively isomorphic to the Lindenbaum algebra of, say, Peano Arithmetic. Since this is a countable boole algebra without
atoms, it follows that this algebra is isomorphic to the free boole algebra on countably many generators. We note that
recursive isomorphism is really about \emph{numbered objects} or \emph{numerations}. Thus, the Lindenbaum algebras of
{\sf J} and of, e.g., {\sf PA} are isomorphic but they cannot, \emph{qua numbered algebras}, be recursively isomorphic, since
{\sf J} is decidable and {\sf PA} is not. 
\end{remark}

\subsection{A First Proof of Theorem~\ref{grotesmurf}}
We prove a theorem that provides, for every essentially undecidable RE theory, a class of theories 
that do not interpret it. The members of this class can instantiate many desirable recursion theoretic
properties. For example, there is a member in any RE degree.

Let $\mathcal X$ be a set of numbers. We say that $W$ is a ${\sf J},\mathcal X$-theory when $W$ is axiomatised
over {\sf J} by boolean combinations of sentences ${\sf A}_s$ for $s\in \mathcal X$.

\begin{theorem}\label{supersmurf}
Consider any essentially undecidable RE theory $U$. Then, we can effectively find a
recursive set $\mathcal X$ \textup(from an index of $U$\textup) such that no consistent ${\sf J},\mathcal X$-theory interprets $U$. 
\end{theorem}

\begin{proof}
Let ${\sf C}_{n,0},\dots, {\sf C}_{n,2^n-1}$ be an enumeration of all conjunctions
of $\pm {\sf A}_i$, for $i<n$.
Suppose $U$ is an essentially undecidable RE theory.
Let $\upsilon_0,\upsilon_1,\dots$ be an effective enumeration of the theorems
of $U$. Let $\tau_0,\tau_1,\dots$ be an effective enumeration of all translations from
the $U$-language into the {\sf J}-language.

Consider $n$, $\tau_i$ and ${\sf C}_{n,j}$, for $j<2^n$.
We claim that, for some $k$, we have ${\sf J}+{\sf C}_{n,j} \nvdash \upsilon^{\tau_i}_k$.
If not, then $\tau_i$ would carry an interpretation of $U$ in ${\sf J}+{\sf C}_{n,j}$.
This contradicts the fact that ${\sf J}+{\sf C}_{n,j}$ is decidable.

Thus, we can effectively find a number $p_{n,i,j}$ as follows.
We find the first $k$ such that ${\sf J}+{\sf C}_{n,j} \nvdash \upsilon^{\tau_i}_k$. 
This can be effectively done since ${\sf J}+{\sf C}_{n,j} $ is decidable. 
Then, we reduce, over {\sf J}, the sentence $\upsilon^{\tau_i}_k$ to a boolean combination of
${\sf A}_s$. Let $p_{n,i,j}$ be the supremum of the $s+1$ such that ${\sf A}_s$ occurs in this boolean combination.

We define $f(n,i)$ to be the maximum of the $p_{n,i,j}$ and $n$.
Let $F(0) := 0$ and let $F(k+1):= f(F(k)+1,k)$. Clearly $F$ is recursive and strictly increasing.
Let $\mathcal X$ be the range of $F$. Clearly, $\mathcal X$ is recursive.

Let $W$ be any consistent ${\sf J},\mathcal X$-theory.
Suppose we would have $K:W \rhd U$.
Let the underlying translation of $K$ be $\tau_{n^\ast}$.

Clearly, there is a $j^\ast$ such that 
$W+{\sf C}_{F(n^\ast)+1,j^\ast}$ is consistent. (This is a non-constructive step.)

By construction, there is a
 $\phi$ with $U \vdash \phi$ and ${\sf J}+{\sf C}_{F(n^\ast)+1,j^\ast}\nvdash \phi^K$, such 
 that  ${\sf J} \vdash \phi^K \iff \chi$, where $\chi$ is a
 boolean combination  of ${\sf A}_s$, where $s<F(n^\ast +1)$.
We note that none of the ${\sf A}_s$ with $F(n^\ast) < s < F(n^\ast+1)$ occurs in the axiomatisation of $W$.
So, there is a ${\sf C}_{F(n^\ast+1),p}$ that extends  ${\sf C}_{F(n^\ast)+1,j^\ast}$ such that ${\sf J}+ {\sf C}_{F(n^\ast+1),p} \vdash \neg\, \phi^K$.
On the other hand, ${\sf C}_{F(n^\ast+1),p}$ is clearly consistent with $W$, by the mutual independence of the ${\sf A}_\ell$.
A contradiction.
\end{proof}

It now follows:

\begin{theorem}\label{aanzienlijkesmurf}
For every essentially undecidable RE theory $U$ there is an essentially undecidable RE theory $W$ such that $W \nrhd U$.
We can find an index for $W$ effectively from an index of $U$.
\end{theorem}

\begin{proof}
Suppose $U$ is an essentially undecidable RE theory.
Let $\mathcal X$ be the recursive set promised for $U$ by Theorem~\ref{supersmurf}.
Let $\mathcal Y,\mathcal Z$ be a pair of recursively inseparable RE sets that are subsets of $\mathcal X$. We define
$W := {\sf J} +\verz{{\sf A}_n \mid n \in \mathcal Y} +   \verz{\neg\, {\sf A}_n \mid n \in \mathcal Z}$. 
Then $W$ is a consistent essentially undecidable RE theory
and $W \nrhd U$.
\end{proof}

Clearly, Theorem~\ref{aanzienlijkesmurf}  gives us Lemma~\ref{hulpsmurf} and, thus, Theorem~\ref{grotesmurf}.

\begin{remark}
It is easily seen that we can give the theory $W$ of Theorem~\ref{supersmurf} many extra properties. For example, it can be
Turing persistent and at the same time of any given non-zero RE degree. See the next section for these notions.
\end{remark}

\subsection{A Result by Shoenfield}
We present some basic ideas from Shoenfield's paper \cite{scho:degr58}.
We will use Shoenfield's result in the proof of Lemma~\ref{hulpsmurf}.

We first need a purely recursion theoretic result. We write $\leq_{\sf T}$ for Turing reducibility.
Our proof is just a minor variation of Shoenfield's proof.

\begin{theorem}[Shoenfield]\label{schonesmurf}
Let $\setvara$ be any recursively enumerable \textup(RE\textup) set with index $a$. Then, we can effectively find 
RE-indices $b$ of a set $\setvarb$ and $c$ of a set $\setvarc$ from $a$,
such that we have:
\begin{enumerate}[i.]
\item
$\setvarb\leq_{\sf T} \setvara$  and $\setvarc \leq_{\sf T} \setvara$.
\item 
 $\setvarb\cap \setvarc = \emptyset$.
 \item
 Suppose $\setvard$ is an RE set that separates $\setvarb$ and $\setvarc$, i.e., $\setvarb\subseteq \setvard$ and $\setvarc \cap \setvard = \emptyset$, 
 then $\setvara \leq_{\sf T} \setvard$.
 \end{enumerate}
\end{theorem}

We  represent $x\in \setvara$ by the formula $\exists  y\, {\sf T}_1(a,x,y)$. We assume that the computation $y$ is unique when it exists.
We write, e.g., $(x)_0 \in \setvara$ for $\exists  y\, {\sf T}_1(a,(x)_0,y)$. We treat other indices similarly.
The notation $\exists x\, \phi \leq \exists y \,\psi$ means
$\exists x\, (\phi \wedge \forall y < x\, \neg\, \psi)$ and 
$\exists x\, \phi < \exists y \,\psi$ means
$\exists x\, (\phi \wedge \forall y \leq x\, \neg\, \psi)$.
%We write $(\chi < \nu)^\top$ for $\nu \leq \chi$.

\begin{proof}
We define:
\begin{itemize}
\item
$x\in {\sf Z}$ iff $\exists z\, {\sf T}_1((x)_1,x,z)$.
\item
$x\in \setvarb$ iff $((x)_0\in \setvara) < (x\in {\sf Z})$.
\item
$x\in \setvarb^\bot$ iff $(x\in {\sf Z})\le ((x)_0\in \setvara)$.
\item
$x\in \setvarc$ iff $(x)_0\in \setvara \wedge x \in \setvarb^\bot$.
\end{itemize}

Claims (i) and (ii) are trivial.\footnote{Note that $\setvarb^\bot$ need not be
Turing reducible to $\setvara$.} 

Consider any $\setvard$ with index $d$. We note that $\tupel{w,d}\in \setvard$ iff $\tupel{w,d} \in {\sf Z}$.
Suppose $\setvard$ separates $\setvarb$ and $\setvarc$.

We  first show $w\in \setvara$ iff $\tupel{w,d}\in \setvarb$.
From right to left is immediate. Suppose 
$w\in \setvara$. Then either $\tupel{w,d}\in \setvarb$ or $\tupel{w,d}\in \setvarc$.
In case  $\tupel{w,d}\in \setvarc$, we find $\tupel{w,d}\in {\sf Z}$, by the definition of $\setvarb^\bot$. Hence
$\tupel{w,d} \in \setvard$. \emph{Quod non}, since $\setvarc$ and $\setvard$ are disjoint. So, $\tupel{w,d}\in \setvarb$.

In case $\tupel{w,d}\in \setvard$, we have $\tupel{w,d}\in {\sf Z}$. So, we can effectively
determine whether $\tupel{w,d}\in \setvarb$ and, thus, whether $w\in \setvara$.
In case $\tupel{w,d}\not\in \setvard$, we have $\tupel{w,d}\not\in \setvarb$, and, hence, $w\not\in \setvara$.
\end{proof}

Let us say that an RE theory $U$ is \emph{Turing-persistent} iff $U$ is consistent and, whenever $U\subseteq V$, where $V$ is RE and consistent, we have
$U \leq_{\sf T} V$. It is easy to see that the Turing-persistence of $U$ implies: if $U \lhd V$, then $U \leq_{\sf T} V$, whenever $V$ is consistent.
We note that, if $U$ is Turing persistent and undecidable, then it is essentially undecidable.

Consider an RE set $\setvara$ with index $a$. Let $\setvarb$ and $\setvarc$ be the sets constructed above.
Let ${\sf sch}(a) := {\sf J}+ \verz{{\sf A}_n \mid n\in \setvarb} + \verz{\neg \,{\sf A}_m\mid m \in \setvarc}$.
We have:

\begin{theorem}[Shoenfield]
The theory ${\sf sch}(a)$ has the same Turing degree as $\mathcal A$. Moreover, the theory
is Turing persistent. It follows that, if $\mathcal A$ is undecidable, then ${\sf sch}(a)$ is essentially
undecidable. Thus, there is an essentially undecidable theory in every
RE, non-recursive Turing degree.
\end{theorem} 

\begin{proof} For any consistent RE extension $V$ of $\mathsf{sch}(a)$, the set $\{i\mid V\vdash \mathsf{A}_i\}$ clearly is an 
RE set separating $\setvarb$ and $\setvarc$ and hence $\setvara$ is Turing reducible to $V$.  It is obvious that 
$\mathsf{sch}(a)$ is consistent. Thus to finish the proof we only need to show that $\mathsf{sch}(a)$ is Turing reducible to $\setvara$.

We describe a $\setvara$-recursive procedure of checking whether given sentence $\varphi$ of the language of {\sf J} is a 
theorem of $\mathsf{sch}(a)$. We can effectively find a boolean combination  $\varphi^\ast$ of ${\sf A}_i$ that is equivalent
over {\sf J} to $\phi$. Let the set of indices of ${\sf A}_i$ occurring in $\phi^\ast$ be $\mathcal I$. 
 Also we can effectively find, using $\setvara$ as an oracle, 
 the conjunction $\psi$ of the  ${\sf A}_j$ with $j \in \mathcal B\cap \mathcal I$ and the
$\neg\, {\sf A}_s$ with $s\in \mathcal C\cap \mathcal I$. Using the mutual independence of the ${\sf A}_n$, it is easy to see that
${\sf sch}(a) \vdash \phi$ iff ${\sf J} \vdash \psi \to \phi^\ast$, and ${\sf J} \vdash \psi \to \phi^\ast$ iff 
$\psi \to \phi^\ast$ is a propositional tautology. Whether $\psi \to \phi^\ast$ is a tautology can be checked with a truth table. 
\end{proof}

\begin{remark}
The result that there is an essentially undecidable theory in every
RE, non-recursive Turing degree is due to Shoenfield. See \cite{scho:degr58}.
This result was improved by Hanf. He proves that there is an essentially undecidable \emph{finitely axiomatised} theory in every
RE, non-recursive Turing degree. See \cite{hanf:mode65}.  By a simple adaptation of the
argument, Hanf's results imply that there is a Turing persistent
finitely axiomatised theory in every RE, non-recursive Turing degree.
\end{remark}

\begin{example}
We provide an example of an essentially undecidable theory that is not Turing persistent.
Consider RE Turing degrees $d$ and $e$ with $0<d<e$. Let $\mathcal X$ be an RE set of degree $d$ and let $\mathcal W$ be an
RE set of degree $e$. We take $\mathcal Y$ and $\mathcal Z$ to be the recursively inseparable RE sets provided by Shoenfield's
result. Let 
\begin{itemize}
\item
$U_0 := {\sf J}+ \verz{{\sf A}_{2i} \mid i \in \mathcal Y} +   \verz{\neg\, {\sf A}_{2j} \mid j \in \mathcal Z}$
\item
$U_1 := U_0+  \verz{{\sf A}_{2k+1} \mid k \in \mathcal W}$.
\item
 $U_2 := U_1+   \verz{{\sf A}_{2k+1} \mid k \in \omega}$.
 \end{itemize}
 Clearly $U_0$ and $U_2$ are Turing persistent and have Turing degree $d$. The theory $U_1$ had degree
 $e$ and, hence, cannot be Turing persistent. Yet, it is clearly
 essentially undecidable. 
 
 So our example shows that
 there are essentially undecidable theories that are not Turing persistent. Moreover, a Turing persistent theory can have a non-Turing persistent 
 RE extension and a non-Turing persistent recursively inseparable theory can have a Turing persistent
 RE extension.
\end{example}

\begin{question}
A theory is \emph{essentially} Turing-persistent if all its consistent RE extensions are Turing-persistent.
Clearly, any Turing-persistent theory in degree $0'$ is essentially Turing-persistent. 
Are there essentially Turing-persistent theories of other degrees?
\end{question}

\subsection{A Second Proof of Theorem~\ref{grotesmurf}}
In this subsection, we prove our main result for
recursively enumerable theories.
We need a basic fact from recursion theory.

Let {\sf Rec} be the set of indices of recursive sets.
We have the following theorem.

\begin{theorem}[Rogers, Mostowski]\label{romo}
{\sf Rec} is complete $\Sigma^0_3$.
\end{theorem}

See \cite[Chapter 14, Theorem XVI]{roge:theo67} or \cite[Corollary 4.3.6]{soar:turi16}.
We also have the following theorem.
We have the following theorem
that is given as exercise~4.3.14 of \cite[Page 91]{soar:turi16}.
Let {\sf Rsep} be the set of pairs of indices of recursively separable RE theories. 
\begin{theorem}\label{soar}
{\sf Rsep} is complete $\Sigma^0_3$.
\end{theorem}

This follows immediately from Theorem~\ref{schonesmurf}, since that gives a reduction of
{\sf Rec} to {\sf Rsep}.

We now prove Theorem~\ref{grotesmurf}.

\begin{proof}[Proof of Theorem~\ref{grotesmurf}]
It is sufficient to prove Lemma~\ref{hulpsmurf}.
Suppose that there is an essentially undecidable recursively enumerable theory
$U^\star$ that is the interpretability minimum. 
So, we have:
\begin{eqnarray*}
a \not \in {\sf Rec} & \text{iff} & {\sf sch}(a) \text{ is essentially undecidable}\\
& \text{iff} & {\sf sch}(a) \rhd U^\star
\end{eqnarray*}
Since, interpretability between recursively enumerable theories is
$\Sigma^0_3$, 
it would follow that {\sf Rec} is $\Pi^0_3$.\footnote{In fact,
interpretability between recursively enumerable theories is
complete $\Sigma^0_3$. See \cite{shav:inter97}.} \emph{Quod non}, by 
Theorem~\ref{romo}.
\end{proof}

Here is a variant of the proof.
\begin{proof}[Variant of the Proof of Theorem~\ref{grotesmurf}]
It is sufficient to prove Lemma~\ref{hulpsmurf}.
Suppose that there is an essentially undecidable recursively enumerable theory
$U^\star$ that is the interpretability minimum. Let $\tupel{a,b}$ be a pair of indices
of RE sets and let $\mathcal X$ and $\mathcal Y$ be the sets defined by $a$, respectively $b$.

We define:
\[ {\sf so}(a,b) := {\sf J}+ \verz{{\sf A}_i \mid i\in \mathcal X} +\verz{\neg\,{\sf A}_j \mid j\in \mathcal Y}.\]
We have: 
\begin{eqnarray*}
\tupel{a,b} \not \in {\sf Rsep} & \text{iff} & {\sf so}(a,b) \text{ is essentially undecidable or inconsistent}\\
& \text{iff} & {\sf so}(a,b) \rhd U^\star
\end{eqnarray*}
Since, interpretability between recursively enumerable theories is
$\Sigma^0_3$, 
it would follow that {\sf Rsep} is $\Pi^0_3$.
 \emph{Quod non}, by 
Theorem~\ref{soar}.
\end{proof}

We note that our result is insensitive for the precise notion of interpretability  used. 
It could very well be that it also works for even more general notions like forcing interpretability. 
However, we did not explore this.

%\bibliographystyle{alpha}
%\bibliography{provint}

\end{document}